\documentclass[12pt]{amsart}
\usepackage[T1]{fontenc}
\usepackage{amssymb}
\newcommand{\cf}{\operatorname{cf}}

\newcommand{\Aee }{\mathcal A}
\newcommand{\Ree }{\mathcal R}
\newcommand{\Bee }{\mathcal B}
\newcommand{\Iee }{\mathcal I}
\newcommand{\Cee }{\mathcal C}
\newcommand{\Hee }{\mathcal H}
\newcommand{\Wee }{\mathcal W}
\newcommand{\Qee }{\mathcal Q}
\newcommand{\See }{\mathcal S}
\newcommand{\Pee }{\mathcal P}
\newcommand{\Fee }{\mathcal F}

\renewcommand{\int}{\operatorname{Int}}

\newtheorem{theorem}{Theorem}
\newtheorem{corollary}[theorem]{Corollary}
\newtheorem{lemma}[theorem]{Lemma}

\newtheorem{fact}[theorem]{Fact}

\author{Piotr Kalemba}
\address{Piotr Kalemba \\
 Institute of Mathematics, University of
Silesia \\
ul. Bankowa 14, 40-007 Katowice}
\email{pkalemba@math.us.edu.pl}
\author{Szymon Plewik}
\address{Szymon Plewik\\Institute of Mathematics,
University of Silesia, ul. Ban\-ko\-wa 14, 40-007 Katowice}
\email{plewik@math.us.edu.pl}
\author{Anna Wojciechowska}
\address{Anna Wojciechowska\\Institute of Mathematics,
University of Silesia, ul. Ban\-ko\-wa 14, 40-007 Katowice}
\email{annawojciechowska@math.us.edu.pl}
\begin{document}

\title{On the ideal $(v^0)$}
\date{}
\begin{abstract}
  The $\sigma$-ideal $(v^0)$ is associated with the Silver forcing, see   \cite{bre}. Also,  it constitutes the family of all completely doughnut null sets, see \cite{hal}. We  introduce segment topologies to state some resemblances of $(v^0)$ to the family of  Ramsey null sets.  To describe $add(v^0)$ we  adopt a proof of Base Matrix Lemma. Consistent results are stated, too.   Halbeisen's conjecture $cov(v^0) = add(v^0)$ is confirmed under the hypothesis  $t=  \min \{\cf (\frak c ),  r\} $. The hypothesis $cov(v^0)=\omega_1$ implies that $(v^0)$ has the ideal type $(\frak c, \omega_1,\frak c)$.
\end{abstract}

\maketitle
\section{Introduction}
Our discussion  focuses around the family $[\omega]^\omega$ of all infinite subsets of natural numbers. We are interested in some  structures on $[\omega]^\omega$ which correspond to the inclusion $\subseteq$ and to the partial order $\subseteq^*$.  Recall that, $A \subseteq^* X$ means that the set $A \setminus X$ is finite. We assume that the readers are familiar with some properties of the partial order $([\omega]^\omega, \subseteq^*)$. For instance,   
gaps of type $(\omega, \omega^*)$ and  $\omega$-limits   do not exist, see F. Hausdorff \cite{hau} or compare F. Rothberger \cite{rot}.  We refer to books \cite{eng} and \cite{ker} for the mathematics used in this note. In particular, one can find basic facts about completely Ramsey sets and its applications to  the descriptive set theory  in \cite{ker} p. 129 - 136. Let us add, that E. Ellentuck (1974) was not the first one who considered properties of the  topology which is called by his name. Non normality of this topology was established by V. M. Ivanowa (1955) and J. Keesling (1970), compare \cite{eng} p. 162 -163. We refer the readers to papers \cite{bs}, \cite{bre}, \cite{ips}, \cite{lou}, \cite{ls} and \cite{nr} for other applications of completely Ramsey sets, not discussed in \cite{ker}. 

  Let $\Wee$ be a family of sets such that $\cup \Wee \notin \Wee$. Recall that,   $$add (\Wee) = \min \{ |\Fee|: \Fee \subseteq \Wee \mbox{ and } 
\cup \Fee \notin \Wee\}$$ is called the \textit{additivity number} of $\Wee$. But   $$cov (\Wee) = \min \{ |\Fee|: \Fee \subseteq \Wee \mbox{ and } 
\cup \Fee = \cup \Wee\} $$ is called the \textit{covering number} of $\Wee$.   
 Thus, $add(v^0)$ and $add(v)$ denote the additivity number of the ideal $(v^0)$ and     of the $\sigma$-field $(v)$, respectively. But $cov(v^0)$ denotes the covering the ideal $(v^0)$. 
 For definitions of the tower number $t$ and the reaping number $r$ we refer to \cite{bla}. One can find there  a thorough discussion of  consistent properties of $t$ and $r$, too.

J. Brendle \cite{bre} considered a few tree-like forcings with  $\sigma$-ideals associated to them. The concept of these ideals is modeled on $s^0$-sets of Marczewski \cite{mar} and Morgan's category base \cite{mor}. One of these ideals is the ideal $(v^0)$. It is associated with the Silver forcing. The ideal $(v^0)$ is examined in papers \cite{bhl}, \cite{hal} and \cite{knw}, too.  L. Halbeisen \cite{hal} found some analogy with completely Ramsey sets and introduced so called completely doughnut sets, i.e. $v$-sets in our terminology. He introduced a pseudo topology - and called it the doughnut topology -  such that $X$ is a $v$-set iff $X$ has the Baire property with respect to the doughnut topology.  Using the method of B. Aniszczyk \cite{ani} and K. Schilling \cite{sch} we introduce segments topologies, each one corresponds to $v$-sets similarly as  Halbeisen's pseudo topology. 
To describe $add(v)$ we  adopt a proof of Base Matrix Lemma, compare \cite{bps} and \cite{bs}.  The height $\kappa(v)$ of a base $v$-matrix equals to $add(v)=add(v^0)$. With a base $v$-matrix it is associated  the increasing family of $v^0$-sets with the union outside the ideal $(v^0)$. We can not confirm (in ZFC) that this union is $[\omega]^\omega$. Therefore, we get a few consistent results. For example,  $cov (v^0)=\omega_1$ implies that $(v^0)$ has the ideal type $(\frak c, \omega_1,\frak c)$. The conjecture of Halbeisen $cov(v^0) = add(v^0)$ is confirmed under $t=  \min \{\cf (\frak c ),  r\} $. 

On the other hand, each maximal chain contained in a base $v$-matrix gives a $(\kappa(v),\kappa(v)^*)$-gap or a $\kappa(v)$-limit. If $cov(v^0) = add(v^0)$, then one can improve any  base $v$-matrix such that each maximal chain,  contained in a new one,   gives a $(\kappa(v),\kappa(v)^*)$-gap, only. But,  whenever $cov(v^0) \not= add(v^0)$, then there exist  $\kappa(v)$-limits. Thus our's research   continue Hausdorff \cite{hau} and Rothberger \cite{rot}, too.

\section{Segments and $*$-segments}
In this section we consider segments and $*$-segments. The facts quoted here immediately  arise from well known ones.
A set 
$$<A,B> = \{ X \in [\omega]^\omega :
A \subseteq X \subseteq B \}$$ 
is called a \textit{segment}, whenever   $A \subseteq B\subseteq \omega$ and   $B\setminus \ A \in  [\omega]^\omega$. 
 By the definition any segment has the cardinality continuum. If \mbox{$<A,B>$} and $<C,D>$ are segments, then the intersection $$<A,B> \cap <C,D>= <A\cup C, B \cap D >$$ is finite or is a segment. It is a segment, whenever $A\cup C \subset  B \cap D $ and $ B \cap D \setminus A\cup C \in[\omega]^\omega$. Thus,  the family of all segments is not closed under finite intersections. 
\begin{fact}\label{a1} Any segment contains continuum many disjoint segments.\end{fact} \begin{proof} Let $<A,B>$ be a segment. Consider a family $\Ree$ of almost disjoint subsets of $B\setminus A$ of the cardinality continuum.   Divide each set $C \in \Ree$ into two infinite subsets $D_C$ and $C\setminus D_C$.  The family  $$\{ <A\cup D_C, A \cup C>: C\in \Ree \}$$ is a desired one. 
\end{proof}

For any set $S \subseteq [\omega]^\omega$ we put  $$S^*= \{Y: X\subseteq^* Y  \subseteq^* X \mbox{ and } X \in S\}.$$ Thus, $S^*$ is a countable union of copies of $S$, i.e. the union of sets $\{(X\setminus y) \cup (y\setminus X): X\in S\}$, where $y\subset \omega$ runs over finite subsets.  If $<A,B>$ is a segment, then the set      $$ \{ X:
A \subseteq^* X \subseteq^* B \}= <A,B>^*$$ is called  $*$-\textit{segment}.

\begin{fact}\label{BBB} If $\{ < A_n, B_n>: n \in \omega \}$ is a   sequence of segments decreasing with respect to the inclusion, then there exists a segment $<C,D>$ such that $< C, D>\subseteq < A_n, B_n>^*$ for each $n\in \omega$.
\end{fact}
\begin{proof} Let $\{ < A_n, B_n>: n \in \omega \}$ be a  decreasing sequence of segments. We have $$A_0 \subseteq A_1 \subseteq A_2 \subseteq \ldots  \subseteq B_2 \subseteq B_1 \subseteq B_0.$$   Choose a set $C\in [\omega]^\omega$ such that $A_n \subseteq^* C \subseteq^* B_n$ for each $n\in \omega$. Additionally,  we can assume that sets $C \setminus A_n$ and $B_n\setminus C$ are infinite,  since there are no $\omega$-limits and $(\omega, \omega^*)$-gaps. Then, choose a set $D\in [\omega]^\omega$ such that $D \setminus C$ is infinite and  $C  \subseteq D \subseteq^* B_n$ for each $n\in \omega$.  
\end{proof}
 Occasionally segments show up in the descriptive set theory. For example, the work of G. Moran and D. Strauss \cite{ms} implies that any subset of $[\omega]^\omega$ having the  property of Baire
 and of second category contains a segment. In other words, it has the doughnut property. One can prove this                                                                           adopting the proof of Proposition 2.2 in \cite{hp}, also. The work \cite{ms} implies that any  subsets of $[\omega]^\omega$ with positive Lebesgue measure contains a segment, compare \cite{pv} and \cite{knw}.

\section{Segment topologies}

  C. Di Prisco and J. Henle \cite{hp} introduced so called  doughnut property. Namely, a subset $S\subseteq [\omega]^\omega$ has the doughnut property, whenever $S$ contains a segment or is disjoint with a segment.  Afterwards, Halbeisen \cite{hal} generalized this property, considering so called completely doughnut sets and completely doughnut null sets. We
feel that the use of "doughnut"  is not appropriate. We swap it onto notations similar to that, which were  used in  \cite{bre} or \cite{knw}. A     
    subset $S \subseteq [\omega]^\omega$ is called a \textit{$v$-set}, if for each segment $<A,B>$
there exists a segment $<C,D>\subseteq <A,B>$ such that $$ <C,D>\subseteq S \mbox{ or } <C,D> \cap S = \emptyset.$$
    If  always holds $  <C,D> \cap S = \emptyset$, then $S$ is called a \textit{$v^0$-set}.  
Any subset of a $v^0$-set is a   $v$-set and a $v^0$-set, too. Also, the complement of a $v$-set is a $v$-set. According to  facts 1.3, 1.5 and 1.6 in Halbeisen \cite{hal},  the family of all $v$-sets is a $\sigma$-field and we denote this field $(v)$. The family of all $v^0$-sets is a $\sigma$-ideal and we denote this ideal $(v^0)$. One can  find many interesting results about   $(v^0)$ in papers \cite{bre},  \cite{bhl} and \cite{knw}.   

We amplify the method of  Aniszczyk \cite{ani} and  Schilling \cite{sch} to introduce some topologies, which correspond to $(v)$. These topologies have  the same features as the pseudo topology, which was considered by  Halbeisen \cite{hal}. Fix  a  transfinite sequence $\{ C_\alpha: \alpha < \mathfrak{c}\}$ consisting   of all segments. Put $V_0 = C_0$.  For every ordinal number $\alpha < \mathfrak{c}$, let $M_\alpha$ be the union of all intersections $ C_{\beta_1} \cap C_{\beta_2} \cap \ldots \cap C_{\beta_n}$ such that  
$$ |C_{\beta_1} \cap C_{\beta_2} \cap \ldots \cap C_{\beta_n}|< \omega, $$ where   
 $ \beta_i \leq \alpha$ and  $1\leq i \leq n$.   Put $V_\alpha = C_\alpha \setminus M_\alpha$.  The topology generated  by all (just defined) sets $V_\alpha$ is called a \textit{segment topology}. There are many segment topologies, since any one depends on an ordering $\{ C_\alpha: \alpha < \mathfrak{c}\}$.  
 We get $|M_\alpha| < \mathfrak{c}$, for any $\alpha < \frak c$. Also, each $V_\alpha$ contains a segment. Therefore,  if $S\subset [\omega]^\omega$ and $|S|<\mathfrak{c}$, then  $S$ is nowhere dense with respect to  any segment topology.  Moreover, we have.
 \begin{lemma} Any family $\{ V_\alpha: \alpha < \mathfrak{c}\}$ is a $\pi$-base and subbase for the segment topology (which it  generates). 
  \end{lemma} 
 \begin{proof}   The family $\{ V_\alpha: \alpha < \mathfrak{c}\}$ is a subbase by the definition. Thus, the family of all intersections $ V_{\beta_1} \cap V_{\beta_2} \cap \ldots \cap V_{\beta_n}$ constitutes  a base. If a base set $ V_{\beta_1} \cap V_{\beta_2} \cap \ldots \cap V_{\beta_n}$ is non-empty, then it has the form of a segment minus a set of the cardinality less than the continuum, exactly   $$ C_{\beta_1} \cap C_{\beta_2} \cap \ldots \cap C_{\beta_n}\setminus ( M_{\beta_1} \cup M_{\beta_2} \cup \ldots \cup M_{\beta_n}).$$ By Fact \ref{a1}, it contains some segment $C_\alpha$. Hence  $ V_{\beta_1} \cap V_{\beta_2} \cap \ldots \cap V_{\beta_n}$ contains some $V_\alpha \subseteq C_\alpha$.
 \end{proof} 
 
 Immediately, one obtains that any two segment topologies determine the same family of nowhere dense sets. As a matter of fact, every element of the base contains a segment and vice versa. Consequently, the nowhere dense sets with respect to any segment topology are the $v^0$-sets.    The next lemma amplifies the fact that there are no  $(\omega, \omega^*)$-gaps. It corresponds to the  result of Moran and Strauss \cite{ms}, compare  Proposition 2.2 in \cite{hp}. We  need the following abbreviation   $$ <A, B>_n = <A, B \setminus( \{0,1, \ldots, n\} \setminus A)>. $$ 
  \begin{lemma}\label{lem3} Let $S_0,  S_1,  \ldots$ be a  sequence of nowhere dense subsets. For any segment $<A,B>$ there exists a segment $<E,F> \subseteq <A,B>$ such that $S_n \cap <E,F>=\emptyset $ for each $n\in \omega$. 
 \end{lemma}
 \begin{proof}  Assume that the sequence  $S_0,  S_1,  \ldots$ is increasing.   
 We shall define points $e_0, e_1, \ldots, e_n$ and sets $$A \subseteq A_0  \subseteq A_1 \subseteq \ldots \subseteq A_n \subseteq B_n \subseteq \ldots \subseteq B_1 \subseteq B_0 \subseteq B,$$ where $B_n \setminus A_n$ is infinite, 
$\{ e_0, e_1, \ldots, e_n\} \subset B_n\setminus A_n$ and   $$ e_n = \min ( B_n \setminus (A_n \cup \{ e_0, e_1, \ldots, e_{n-1}\});$$ and  such that $ <A_n		\cup x, B_n>_{e_n} \cap   S_n =\emptyset$, for each $x \subseteq \{e_0, e_1, \ldots, e_n\} $ and any $n< \omega$.   

We proceed inductively with respect to $n$.
 Let  $e_0 = \min (B\setminus A)$.  Choose a segment $<A^0_0,B^0_0> \subseteq <A,B>_{e_0}\setminus S_0.$ Then, choose sets $A_0 \supseteq A^0_0$ and $B_0 \subseteq B^0_0
\cup \{e_0\}$ such that $e_0 \in B_0\setminus A_0$ and the segment $<A_0\cup \{e_0\},B_0>_{e_0}$ is disjoint with $ S_0 $. We get $$ (<A_0\cup \{e_0\},B_0>_{e_0}\cup <A_0,B_0>_{e_0})\cap S_0 =\emptyset.$$ 
Assume that sets $A_n$ and $B_n$ are defined. Let $$ e_n = \min ( B_n \setminus (A_n \cup \{ e_0, e_1, \ldots, e_{n-1}\})).$$ Enumerate all subsets of $\{e_0,e_1, \ldots, e_n\}$ into  a  sequence $x_1, x_2, \ldots ,x_{2^{n+1}}$.    Choose a segment $$<A^1_n,B^1_n > \subseteq <A_n\cup x_1,B_n >_{e_n}\setminus S_n.$$ If a segment $<A^{k-1}_n,B^{k-1}_n >$ has been already defined,  then choose sets $A^k_n \supseteq A^{k-1}_n$ and $B^k_n \subseteq B^{k-1}_n
\cup \{e_0,e_1, \ldots, e_n\}$ such that $\{e_0, e_1, \ldots, e_n\}  \subset B^k_n\setminus A^k_n$ and the segment $<A^k_n\cup x_k,B^k_n>_{e_n}$ is disjoint with $ S_n $. Let $B_{n+1}$ be the last $B^k_n$ and $A_{n+1}$ be the last $A^k_n$. By the definition, we get $\{e_k: k <\omega \} \subset B_n\setminus A_n$ and $$ \cup\{ <A^k_n\cup x_k,B^k_n>_{e_n}: 0<  k \leq2^{n+1}\} \cap S_n =\emptyset,$$ for any $n<\omega$. 
 Finally, the segment $$ <E,F>=<\cup \{ A_n: n \in \omega\}, \cup \{ A_n: n \in \omega\} \cup \{e_n: n \in \omega\}>$$ is disjoint with each $ S_k$. Indeed, suppose $C\in <E,F>\cap S_k.$ Let $x= C \cap \{e_0, e_1, \ldots , e_k\}$. Then $C\in <A_k\cup x, B_k>_{e_k}$. But this contradicts   $  <A_k\cup x,B_k>_{e_k} \cap S_k =\emptyset.$   \end{proof}
 \begin{corollary}\label{pr3} For any segment topology, the intersection of countable many open and dense sets contains an open and dense subset. \hfill $\Box$
 \end{corollary}

 \begin{corollary}\label{CC2}  The  ideal $(v^0)$  coincides with the family of all  sets of the first category with respect to any segment topology. \hfill $\Box$
\end{corollary}

Recall that, a subset $Y$ of a topological space  $X$ has the  property of Baire whenever  $Y= (G\setminus F)\cup H$, where $G$ is open and   $F$, $H$ are of the first category.   If $X=[\omega]^\omega$ is equipped with a segment topology, then $Y\subseteq X$ has the Baire property (i.e. the  property of Baire with respect to this segment topology) whenever  $Y= G\cup H$, where $G$ is open  and  $H$ is a $v^0$-set. 

\begin{theorem}\label{CC1} The $\sigma$-field  $(v)$   coincides with the family of all sets which have the Baire property with respect to a segment topology.  
\end{theorem} 
\begin{proof} 
Fix a segment topology and  a $v$-set $X$. Let $U=\cup \{V_\beta: V_\beta \subseteq X \} $ and  $W=\cup \{V_\beta: V_\beta \cap X =\emptyset\} $. The union $U\cup W$ is open and dense. Thus  $X= U \cup F$, where  $F\subseteq [\omega]^\omega \setminus (U \cup W)$ is nowhere dense.

 We shall show that any open set is a $v$-set. Suppose a set $X$ is open. Take an arbitrary segment $<A,B>$ and choose a subbase set $V_\alpha \subseteq <A,B>$. There exists $V_\beta \subseteq V_\alpha$ such that $V_\beta \subseteq X$ or  $V_\beta \subseteq \int ([\omega]^\omega \setminus X)$. Each segment  $<C,D> \subseteq V_\beta$ witnesses that $X$ is a $v$-set.
\end{proof} 

Every classical analytic set belongs to $(v)$. This is a counterpart  of Mathias-Silver theorem - compare (21.9) or (29.8) in \cite{ker} - which arises from Halbeisen's paper \cite{hal}. In fact,  one could conclude it similarly like in  the paper by Pawlikowski \cite{paw}.  This was noted   by  Brendle,  Halbeisen and  Löwe in \cite{bhl}. We obtain the counterpart directly, using Theorem \ref{CC2} and theorems (29.11), (29.13) in \cite{ker}.

\section{Base $v$-matrix}
 We shall adopt a proof of  Base Matrix Lemma -  see B. Balcar J. Pelant and P. Simon, compare \cite{bps} and \cite{bs}. 
 There are known some generalizations of this theorem for some partial orders, e.g. compare \cite{mac}. For completeness, we  prove our's version directly. 
If $<A,B>$ and $<C,D>$ are segments, then the intersection $<A,B>^* \cap <C,D>^*$ is countable or  has the  cardinality continuum. In the second case the intersection  is a $*$-\textit{segment}. \\ \indent Whenever $<A,B>^* \cap <C,D>^*$ is countable, then   $<A,B>^*$ and $<C,D>^*$ are called $*$-\textit{disjoint}. \\ \indent 
\begin{lemma} \label{AAA}  If $S$ is a $v^0$-set, then for any segment $<A,B>$ there exists a segment  $<C,D> \subseteq <A,B>$ such that $<C,D>^* \cap S^* = \emptyset$.
\end{lemma}
\begin{proof}
 By the definition, $S^*$ is a countable union of elements of  $(v^0)$, hence  $S^* \in (v^0)$. Thus, any segment $<C,D> \subseteq <A,B>$ disjoint with  $S^*$ is a desired one.
\end{proof}
A family $\Pee$ of $*$-segments  is a $v$-\textit{partition}, whenever  any two distinct members of $\Pee$ are $*$-disjoint and $\Pee$ is maximal with respect to the inclusion. A collection of $v$-partitions is called $v$-\textit{matrix}. A $v$-partition $\Pee$ \textit{refines} a  $v$-partition $\Qee$ (briefly $\Pee\prec \Qee$), if for each $<A,B>^*\in \Pee$ there exists $<C,D>^*\in \Qee$ such that $<A,B >^*\subseteq <C,D>^*$.
 A $v$-matrix $\Hee$ is called \textit{shattering}, if for each $*$-segment $<A,B>^*$ there exists  $\Pee \in \Hee$ and $<A_1,B_1>^* ,<A_2,B_2>^* \in \Pee$ such that $<A_1,B_1>^* \cap <A,B>^*$ and $<A_2,B_2>^* \cap <A,B>^*$ are different $*$-segments. Denote by $\kappa (v)$ the least cardinality of a shattering $v$-matrix.  
\begin{lemma}\label{EEE} If  a  $v$-matrix $\Hee$ is of the cardinality less than $\kappa (v)$, then there exists a $v$-partition $\Pee$ which refines any $v$-partition $\Qee \in \Hee$.
\end{lemma} 
\begin{proof} Fix a segment $<A,B>$. Let $\Hee(A,B)= \{ \Pee (A,B): \Pee \in \Hee\}$ be the relative  $v$-matrix such that each $\Pee (A,B)$ consists of all $*$-segments  $ <C,D>^*\cap <A,B>^*$, where $ <C,D>^* \in \Pee $. Any segment $<C,D>$
is isomorphic to  $[D\setminus C]^{\leqslant\omega}$ and  $[\omega]^{\leqslant\omega}$, hence $\Hee(A,B)$ is not shattering relative to $<A,B>^*$. Choose a segment $<C,D> \subseteq <A,B>$ such that   there exists $<E, F>^* \in \Pee$ with $ <C,D>^*\subseteq <E, F>^*$ for every $\Pee \in \Hee$.  
Any  $v$-partition $\Pee$ consisting of above defined $*$-segments \mbox{$<C,D>^*$} is a desired one.
\end{proof}

 Let $h$ be the height of the base matrix . See \cite{bps} and \cite{bs} for rudimentary properties of the cardinal number $h$.
\begin{theorem} \label{tco3} $\omega_1 \leq \kappa (v)\leq h $ and $\kappa (v)$ is a regular  cardinal number.
\end{theorem}
\begin{proof}
Suppose $h < \kappa (v)$. Take a base matrix  $\{\Hee_\alpha: \alpha < h\}$  such as in  2.11 Base Matrix Lemma in \cite{bps}.   Let $\Pee_\alpha$ be a $v$-partition such that for any $<A,B>^* \in \Pee_\alpha $ there exists $V\in \Hee_\alpha$ with $B \setminus A \subseteq^* V$. The $v$-matrix $\{\Pee_\alpha: \alpha < h\}$ contradicts Lemma \ref{EEE}.

 Consider a shattering $v$-matrix $\Hee =\{ \Pee_\alpha: \alpha < \kappa (v) \} $. By Lemma \ref{EEE}, we can assume that $\alpha <\beta$ implies $\Pee_\beta \prec \Pee_\alpha$. Any cofinal family of $v$-partitions from $\Hee$ constitutes a shattering $v$-matrix. Hence $\kappa (v)$ has to be regular.  It is uncountable by Fact \ref{BBB}. 
\end{proof}

\begin{theorem} \label{12} There exists a $v$-matrix $\Hee =\{ \Pee_\alpha: \alpha < \kappa (v) \} $ which is well ordered by the inverse of
$\prec$. Moreover, for
each $*$-segment $<A,B>^*$  there is $ <C,D>^*\in \cup \Hee$ such that  $<C,D>^* \subseteq <A,B>^*$. 
\end{theorem}
\begin{proof} Build  a shattering $v$-matrix $\Hee =\{ \Pee_\alpha: \alpha < \kappa (v) \} $ such that  $\alpha <\beta$ implies $\Pee_\beta \prec \Pee_\alpha$. 
Let $J^c(\Pee_\alpha)$ be the family of all $*$-segments $<A,B>^*$ for  which there are continuum many elements of $ \Pee_\alpha$ not $*$-disjoint with $<A,B>^*$. Let $F:J^c(\Pee_\alpha)\to \Pee_\alpha$ be a one-to-one function such that  $F(G)\cap G$ is a $*$-segment, for every $G\in J^c(\Pee_\alpha)$.  Choose a $v$-partition $$\Qee \supseteq \{  F(G)\cap G: G\in J^c(\Pee_\alpha)\}.$$ Having these, one can improve $\Hee$ to obtain  $\Pee_{\alpha +1} \prec \Qee$ and $\Pee_{\alpha +1} \prec \Pee_\alpha$. One obtains that, if $<A,B>^*\in J^c(\Pee_\alpha)$, then there is $<C,D>^*\in \Pee_{\alpha +1}$ with $<C,D>^*\subseteq <A,B>^*$.

 For each $*$-segment $<A,B>^*$ there exists $\alpha < \kappa (v)$ such that $<A,B>^*\in J^c(\Pee_\alpha)$. Indeed, fix a $*$-segment $<A,B>^*$. Let $B^0_{\alpha_0}$ and $B^1_{\alpha_0}$ be two different $*$-segments belonging to $\Pee_{\alpha_0}$ such that  $D^0_{\alpha_0}=<A,B>^*\cap B^0_{\alpha_0}$ and $D^1_{\alpha_0}=<A,B>^*\cap B^1_{\alpha_0}$ are $*$-segments. Thus, 
  $D^{i_0}_{\alpha_0}\subseteq <A,B>^*$ for $i_0 \in \{0,1\}$. Inductively, 
 let $B^{i_0i_1 \ldots i_{n-1} 0}_{\alpha_n}$ and 
$B^{i_0i_1 \ldots i_{n-1} 1}_{\alpha_n}$ be two different $*$-segments belonging to $\Pee_{\alpha_n}$ such that  $D^{i_0i_1 \ldots i_{n-1} 0}_{\alpha_n}=<A,B>^*\cap B^{i_0i_1 \ldots i_{n-1} 0}_{\alpha_n}$ and $D^{i_0i_1 \ldots i_{n-1} 1}_{\alpha_n}=<A,B>^*\cap B^{i_0i_1 \ldots i_{n-1} 1}_{\alpha_n}$ are $*$-segments. We get  
 $$D^{i_0i_1 \ldots i_{n}}_{\alpha_{n}}\subset D^{i_0i_1 \ldots i_{n-1}}_{ \alpha_{n-1}} \subset <A,B>^*.$$ Put $\beta =\sup \{\alpha_n: n\in \omega\}$. By the construction and  Fact \ref{BBB}, we get $<A,B>^*\in J^c(\Pee_{\beta +1})$. Therefore, for each $*$-segment $<A,B>^*$ there exists $\alpha < \kappa (v)$ and 
$<C,D>^* \in \Pee_\alpha$ such that  $<C,D>^*\subseteq <A,B>^*$
\end{proof}

 Let  $\{ \Pee_\alpha: \alpha < \kappa (v) \} $ be a $v$-matrix as in the Theorem \ref{12}. In general, any two members of the union $\cup \{\Pee_\alpha:  \alpha < \kappa (v)\}$ are $*$-disjoint or one is included in the other.
 One  could remove a set $M_C$ of cardinality less than $\frak c$ from each $*$-segment $C \in \cup \{\Pee_\alpha:  \alpha < \kappa (v)\}$   such that any two members of the family $$\Qee = \{C\setminus M_C: C \in \cup \{\Pee_\alpha:  \alpha < \kappa (v)\}\}$$ are disjoint or one is included in the other. Any $\Qee$ as above is called  \textit{a base $v$-matrix}. Thus, $\kappa(v)$ is the height of a base $v$-matrix.  The next theorem yields analogy to nowhere Ramsey sets, compare \cite{ple} p. 665.  
  \begin{theorem} The ideal $(v^0)$ coincides with the family of all nowhere dense subsets with respect to the topology generated by a base $v$-matrix. \end{theorem}
\begin{proof} Let $S\subseteq [\omega]^\omega$ be a $v^0$-set and $\Qee$ a base $v$-matrix. Any set $W\in \Qee$ is a $*$-segment minus a set of cardinality less than $\frak c$. By Fact  \ref{a1} and Lemma \ref{AAA},  there is a $*$-segment $<A,B>^*\subseteq W$ such that  $<A,B>^* \cap S = \emptyset$, for each $W\in \Qee$. By Theorem \ref{12} there exists a $*$-segment $V \in \cup \{\Pee_\alpha:  \alpha < \kappa (v)\}$ such that $V \subseteq <A,B>^*$.   Sets $V\setminus M_V \in \Qee $ witnesses  that $S$ is nowhere dense.

Let $S$ be a nowhere dense set. Take a segment $<A,B>$. Choose  a $*$-segment $W\in \cup \{\Pee_\alpha:  \alpha < \kappa (v)\}$ such that $W \subseteq <A,B>^*$. Then choose $V\in \Qee$ such that $V \subseteq W \setminus  S$. Any segment  $<C,D>\subseteq V$ witnesses that  $S$ is a $v^0$- set.
\end{proof}

In ZFC,  Hausdorff \cite{hau} proved that there exists a $(\omega_1,\omega_1^*)$-gap. This suggests that the height of a base $v$-matrix could be $\omega_1$. We do not know: $$\mbox{ Is it  consistent  that $\omega_1 \not= \kappa(v)$?} $$ Without loss of generality, one can add to the definition of a base $v$-matrix that $\Pee_\beta  \prec \Pee_\alpha$ means that for each $<C,D>^* \in \Pee_\beta$ there exists $<A,B>^*\in \Pee_\alpha$ such that  $<C,D> \subset <A,B>$ and sets $C\setminus A$, $B\setminus D$ are infinite. This yields that 
each maximal chain contained in a such base $v$-matrix produces a $(\kappa(v),\kappa(v)^*)$-gap or a $\kappa(v)$-limit.  We need $add (v^0) =
cov (v^0)$  to obtain a  base $v$-matrix  such that each maximal chain contained in it   produces a $(\kappa(v),\kappa(v)^*)$-gap, only.    So, we  consider additivity and covering numbers of the ideal $(v^0)$.
\section{Additivity and covering numbers}
Foreseeing a counterpart of  Plewik's result that the  additivity number of completely Ramsey sets equals to the covering number of Ramsey null sets - compare \cite{bs} p. 352 - 353 -  Halbeisen  set the following question at the end  of  \cite{hal}:  Does
  $$ add (v^0) =
cov (v^0)?$$
 The answer is obvious under the Continuum Hypothesis.  We add another consistent hypotheses which confirm this equality. 
\begin{lemma}\label{CCC} If $\Pee$ is a $v$-partition, then the complement of the union $\cup \Pee$ is a $v^0$-set. 
\end{lemma} 
\begin{proof} Take a segment $<A,B>$. Since $\Pee$ is maximal, there exists $<C,D>^*\in \Pee$  such that $<A\cup C,B\cap D>^*$ is a $*$-segment contained in $\cup \Pee$. \end{proof}
\begin{lemma}\label{DDD} If $S\subseteq [\omega]^\omega$ is a $v^0$-set,  then there exists   a $v$-partition $\Pee$ such that $\cup \Pee\cap S = \emptyset$. 
\end{lemma} 
\begin{proof} If $S$ is a $v^0$-set, then $S^*$ is a $v^0$-set, too. Thus, for any  segment $<A,B>$ there exists a segment $<C,D>\subseteq <A,B>$  such that $<C,D>^*\cap S^* =\emptyset $. Any  $v$-partition $\Pee$ consisting of a such $<C,D>^*$ is a desired one. \end{proof}

\begin{theorem} \label{15} $\kappa(v) = add(v^0)$.
\end{theorem}  
\begin{proof} Consider a family $\Fee$ of $v^0$-sets such that $|\Fee| < \kappa(v)$. Using Lemma \ref{DDD}, fix a $v$-partition $\Pee_W$ such that $\cup \Pee_W \cap W =\emptyset$ for each $W \in \Fee$. Let $\Pee$ be a $v$-partition  refining any $\Pee_W$, which exists by Lemma \ref{EEE}. The $v^0$-set $[\omega]^\omega \setminus \cup \Pee$ contains $\cup \Fee$.

 Take a base $v$-matrix $\Qee=\{C\setminus M_C: C \in \cup \{\Pee_\alpha:  \alpha < \kappa (v)\}\}$.  Without loss of generality one can assume that for every $C\in \Pee_\alpha$ the difference $C\setminus \cup \Pee_{\alpha +1}$ is not empty.  Then, no segment is disjoint with the union of all sets $[\omega]^\omega \setminus \cup \Pee_\alpha$. In other words, this union is not a $v^0$-set.  Therefore, $\kappa(v) \geq add(v^0)$.
\end{proof}

There are $\sigma$-fields with additivity strictly less than  additivity of its natural $\sigma$-ideal. For example, consider a collection $\Fee$ of $\omega_1$ pairwise disjoint sets, each of the cardinality $\omega_2$.  Let $\See$ be the $\sigma$-field generated by $\Fee$ and all subsets of $\cup \Fee$ of cardinality at most $\omega_1$. Then $add(\See)=\omega_1$ and $add(\{X \in\See: |X| < \omega_2\})=\omega_2$. This is not a case for the field $(v)$.

\begin{theorem} \label{16} $add(v^0) = add(v)$.
\end{theorem}  
\begin{proof} Take a family $\Wee$  witnesses $add(v)$ and fix a segment topology. Each set $W\in \Wee$ is a $v$ set, hence has the form $W= V_W \cup H_W$ where $V_W$ is open and $H_W$ is a $v^0$-set. The union $\cup \{H_W: W \in \Wee \}$ witnesses $add(v^0)$. 

To prove the opposite inequality,  take a  set $\Bee\subseteq [\omega]^\omega$ which is dense and co-dense in a segment topology. One can construct $\Bee$ analogously to the classical construction of a Bernstein set. Let  $\Qee=\{C\setminus M_C: C \in \cup \{\Pee_\alpha:  \alpha < \kappa (v)\}\}$ be a base $v$-matrix . Then, the union of all sets $[\omega]^\omega \setminus \cup \Pee_\alpha$ is not a $v^0$-set. If also, it is  not a $v$-set, then it witnesses $\kappa(v) \geq add(v^0)$. But if this union is a $v$-set, then  sets  $\Bee \setminus \cup \Pee_\alpha$ constitute the family which witnesses $\kappa(v) \geq add(v^0)$.
\end{proof}
  
    Brendle observed that $cov(v^0) \leq r$, see  Lemma 3 in \cite{bre} at page 21. Therefore, we get the following.
\begin{theorem}\label{CCC} $\omega_1\leq \kappa (v) = add (v^0) = add(v) \leq cov(v^0) \leq \min \{\cf (\frak c ), r\}.$
\end{theorem} \begin{proof}
Suppose $[\omega]^\omega = \cup\{\Aee_\alpha: \alpha < cf(\frak c)\},$ where always $|\Aee_\alpha| <\frak c.$ So, $ cov(v^0) \leq cf(\frak c)$, since  each $\Aee_\alpha$ is a $v^0$-set.  Theorems \ref{tco3}, \ref{15}, \ref{16} and Brendle's observation imply the rest inequalities. 
  \end{proof}
 Immediately, we infer the following:  
  If $\kappa (v)= \min \{\cf (\frak c ),  r\} $, then $$\kappa (v)= add (v) =
cov (v^0)= add (v^0). $$ But, if $\kappa (v) < t$, then   
 there are no $\kappa$-limits, see \cite{rot}, and  for any base $v$-matrix $\Qee=\{C\setminus M_C: C \in \cup \{\Pee_\alpha:  \alpha < \kappa (v)\}\}$ the intersection $\cap  \{\cup \Pee_\alpha:  \alpha < \kappa (v)\}$
is empty. This yields $add(v)=cov(v^0)$. Therefore, $t=\min \{\cf (\frak c ), r\}$ implies $add(v)=cov(v^0)$, too. 
\section{Ideal type of $(v^0)$} 
The notion   of an ideal type $(\lambda, \tau, \gamma)$ was introduced in \cite{ple}, where it was obtained some consistent isomorphisms,  applying  the ideal type $(\frak{c}, h, \frak{c})$ to  families of Ramsey null sets. Recall the notion  of ideal types at two steps. To present it in a organized manner we enumerate conditions which are used in  the definition. 

 Firstly,  we   adapt  Base Matrix Lemma \cite{bs}. 
Suppose $\Iee$ is a proper ideal on $\cup \Iee$. A collection of families $\Hee = \{\Pee_\alpha: \alpha < \kappa (\Iee)\}$ is called a \textit{base $\Iee$-matrix} whenever: \\ \indent (1)  
Each family $\Pee_\alpha$ consists of pairwise disjoint subsets of $\cup \Iee$; \\ \indent (2)
 If $\beta <\alpha$, then $\Pee_\alpha$ refines $\Pee_\beta$; \\ \indent (3)
Always $\cup \Iee \setminus \cup\Pee_\alpha $  belongs to $\Iee$; \\ \indent (4) 
  $\Iee$ is the ideal of nowhere dense sets with respect to  the topology generated by $\cup \Hee$.
 
  Secondly,  we   prepare the notions for applications with    Ramsey null sets and $v^0$-sets. The ideal $\Iee$ has the ideal type $(\lambda, \kappa (\Iee), \gamma)$ whenever there exists a base $\Iee$-matrix $\Hee= \{\Pee_\alpha: \alpha < \kappa (\Iee)\}$ such that:  \\ \indent (5)
 Each $\Pee_\alpha$ has the cardinality $\lambda$;\\ \indent (6) If $\beta <\alpha$ and $X\in \Pee_\beta$, then $X\setminus \cup \Pee_\alpha$ has the cardinality $\gamma$;\\ \indent (7) If $\beta <\alpha$ and $Y\in \Pee_\beta$, then $Y$ contains $\lambda$ many members of $\Pee_\alpha$;\\ \indent (8) There are no short maximal chains in $\cup \Hee$, i.e. if $\Cee \subseteq \cup \Hee$ is a maximal chain, then $\Cee\cap \Pee_\alpha$ is  nonempty for each $\alpha < \kappa(\Iee)$; \\ \indent (9) The intersection $\cap\{ \cup\Pee_\alpha: \alpha<\kappa (\Iee)\}$ is empty.

To describe  the ideal type of $(v^0)$ we have to assume that $cov (v^0) = \omega_1$. We do not know: $$ \mbox{ Is it  consistent  that $\omega_1 \not= cov(v^0)$? } $$ If $\omega_1= \min \{\cf (\frak c ), r\}$, then Theorem \ref{CCC} yields 
$\omega_1 = cov(v^0)$. 
\begin{theorem} If $\omega_1 = cov(v^0)$, then $(v^0)$ has the ideal type $(\frak c, \omega_1, \frak c)$.  
\end{theorem}
\begin{proof} Let $\Hee= \{\Pee_\alpha: \alpha < \omega_1\}$ be a base $v$-matrix.  Since $\omega_1 = cov(v^0)$ one can inductively change $\Hee$ such that   $\cap\{ \cup\Pee_\alpha: \alpha<\kappa (v)= \omega_1\} = \emptyset$. If one considers families $\Pee_\alpha$ for limit ordinals, the one obtains a base $v$-matrix which witnesses that $(v^0)$ has the ideal type $(\frak c, \omega_1, \frak c)$. 
\end{proof}

Thus, by \cite{ple} Theorem 2, if $h = \omega_1 = cov (v^0)$, then the ideal $(v^0)$  is isomorphic  with the ideal of all Ramsey null sets. This isomorphism clarify resemblances between definitions of completely Ramsey sets and $v$-sets. However, the $\sigma$-field  $(v)$  and the $\sigma$-field  of all completely Ramsey sets are different. Some Ramsey null sets can be no $v$-sets,   e.g. any intersection of a segment with a  set  which is dense and co-dense in a segment topology. Conversely, some $v^0$-sets  can be no completely Ramsey sets. Indeed, if $\Hee$ is a base matrix, see \cite{bps}, then $(\cup \Hee)^*$ is not a completely Ramsey set and one can check that $(\cup \Hee)^*$ is a $v^0$-set, compare  Brendle \cite{bre}.

\textbf{Acknowledgment}. 
We want to express our gratitude to the referee  for his or her valuable suggestions and  helpful comments.

\end{document}